\documentclass[article]{elsarticle}

\usepackage{amsmath,amssymb,amsthm,algorithmic,algorithm,mathrsfs}
\usepackage{bbm}
\usepackage{vmargin}
\usepackage{natbib}
\usepackage{color}
\usepackage[displaymath]{lineno}

\DeclareMathOperator*{\argmax}{arg\,max}

\setmargins{3.5cm}{2cm}{15cm}{23cm}{.5cm}{.5cm}{.5cm}{2cm}

\definecolor{DB}{rgb}{0,0.07,0.7}

\newcommand{\E}{\mathbb{E}} 
\newcommand{\N}{\mathbb{N}}
\renewcommand{\P}{\mathbb{P}} 

\newcommand{\Z}{\mathbb{Z}}

\newcommand{\V}{\mathcal{V}}
\newcommand{\1}{\mathbbm{1}}
\newcommand{\given}{\,|\,}
\newcommand{\sft}{\mathscr{T}}
\newcommand{\ML}{{\textup{ML}}}

\newcommand{\XX}[2]{\mathcal{X}_{#1}^{#2}(X_1^n)}

\def\F{\mathcal{F}}

\def\ph{\hat{p}}
\def\x{\xi}
\def\X{\chi}

\def\ph{\hat{p}}

\def\C{\mathcal{C}}

\newtheorem{theorem}[equation]{Theorem}
\newtheorem{lemma}[equation]{Lemma}
\newtheorem{prop}[equation]{Proposition}

\newtheorem{assumption}{Assumption}

\numberwithin{equation}{section}

\newtheorem{rem}[equation]{Remark}
\theoremstyle{definition}
\newtheorem{definition}[equation]{Definition}

\title{Context Tree Selection: A Unifying View}

\author[rvt]{A.~Garivier\corref{cor1}} % choose who is the corresponding author
\ead{aurelien.garivier@telecom-paristech.fr}
\author[focal]{F.~Leonardi} % choose who is the corresponding author
\ead{florencia@usp.br}

\cortext[cor1]{Corresponding author}
\address[rvt]{aurelien.garivier@telecom-paristech.fr, LTCI, CNRS, Telecom ParisTech}
\address[focal]{florencia@usp.br, Instituto de Matem\'atica e Estat\'\i stica, Universidade de S\~ao Paulo}
\begin{document}
\begin{abstract}

Context tree models have been introduced by Rissanen in \cite{rissanen1983} as a parsimonious generalization of Markov models. Since then, they have been widely used in applied probability and statistics.
The present paper investigates non-asymptotic properties of two popular procedures of context tree estimation: Rissanen's algorithm Context and penalized maximum likelihood. First showing how they are related, we prove finite horizon bounds for the probability of over- and under-estimation. 
Concerning over-estimation, no boundedness or loss-of-memory conditions are required: the proof relies on new deviation inequalities for empirical probabilities of independent interest. The under-estimation properties rely on classical hypotheses for processes of infinite memory. These results improve on and generalize the bounds obtained in \cite{duarte2006,galves2006,galves2008,leonardi2009}, refining asymptotic results of~\cite{buhlmann1999,csiszar2006}.

\end{abstract}

\begin{keyword}
algorithm Context \sep penalized maximum likelihood \sep model selection \sep variable length Markov chains \sep Bayesian information criterion \sep deviation inequalities  
\end{keyword}
\maketitle

%%%%%%%%%%%%%%%%%%%%%%%%%%
\section{Introduction}
%%%%%%%%%%%%%%%%%%%%%%%%%%
Context tree models (CTM), first introduced by Jorma Rissanen in \cite{rissanen1983} as efficient tools in Information Theory, have been successfully studied and used since then in many fields of Probability and Statistics, including Bioinformatics \cite{bejerano2001a,busch2009}, Universal Coding \cite{willems1995}, Mathematical Statistics \cite{buhlmann1999} or Linguistics \cite{galves2009}.
Sometimes also called Variable Length Markov Chain (VLMC), a context tree process is informally defined as a  Markov chain whose memory length depends on past symbols. 
This property  makes it possible to represent the set of memory sequences as a tree, called the \emph{context tree} of the process.

A remarkable tradeoff between expressivity and simplicity explains this success: no more difficult to handle than Markov chains, they appear to be much more flexible and parsimonious, including memory only where necessary.
Not only do they provide more efficient models for fitting the data: it appears also that, in many applications, the shape of the context tree has a natural and informative interpretation. In Bioinformatics, the contexts trees of a sample have been useful to test the relevance of protein families databases \cite{busch2009} and in Linguistics, tree estimation highlights structural discrepancies between Brazilian and European Portuguese~\cite{galves2009}.

Of course, practical use of CTM requires the possibility of constructing efficient estimators of the model $T_0$ generating the data. It could be feared that, as a counterpart of the model multiplicity, increased difficulty would be encountered in model selection. Actually, this is not the case, and soon several procedures have been proposed and proved to be consistent. 
Roughly speaking, two families of context tree estimators are available. The first family, derived  from the so-called algorithm  \emph{Context} introduced by Rissanen in \cite{rissanen1983}, is based on the idea of \emph{tree pruning}. They are somewhat reminiscent of the CART \cite{breiman:al:84:cart} pruning procedures: a measure of discrepancy between a node's children determines whether they have to be removed from the tree or not.
The second family of estimators are based on a classical approach of mathematical statistics: \emph{Penalized Maximum Likelihood} (PML). For each possible model, a criterion is computed which balances the quality of fit and the complexity of the model. In the framework of Information Theory, these procedures can be interpreted as derivations of the \emph{Minimum Description Length} principle~\cite{barron1998}.

In the case of bounded memory processes, the problem of consistent estimation is clear: an estimator $\hat{T}$ is strongly consistent if it is  equal to $T_0$ eventually almost surely as the sample size grows to infinity. As soon as 1983, Rissanen proved consistency results for the algorithm  Context  in this case.
But later, the possibility of handling infinite memory processes was also addressed. In \cite{csiszar2006}, an estimator $\hat{T}$ is called \emph{strongly consistent} if for every positive integer $K$, its truncation $\hat{T}_{|K}$ at level $K$ is equal to the truncation $T_0{|_K}$ of $T_0$  eventually almost surely. With this definition, PML estimators are shown to be strongly consistent if the penalties are appropriately chosen and if the maximization is restricted to a proper set of models. This last restriction was proven to be unnecessary in the finite memory case~\cite{garivier2006}.

More recently, the problem of deriving \emph{non-asymptotic} bounds for the probability of incorrect estimation was considered. In \cite{galves2006}, non-universal inequalities were  derived for a version of the algorithm Context in the case of finite context trees. These results were generalized to the case of infinite trees in \cite{galves2008}, and to PML estimators in \cite{leonardi2009}. Using recent advances in weak dependence theory, all these results strongly rely on mixing hypotheses of the process.

For all these results, a distinction has to be made between two potential errors: under- and over-estimation. A context of $T_0$ is said to be \emph{under-estimated} if one of its proper suffixes appears in the estimated tree $\hat{T}$,  whereas it is called \emph{over-estimated} if it appears as an internal node of $\hat{T}$. Over- and under-estimation appear to be of different natures: while under-estimation is eventually avoided by the existence of a strictly positive distance between a process and all processes with strictly smaller context trees, controlling over-estimation requires bounds on the fluctuations of empirical processes.

In this article, we present a unified analysis of the two families of context tree estimators. We contribute to a completely non-asymptotic analysis: we show that for appropriate parameters and measure of discrepancy, the PML estimator is always smaller than the estimator given by the algorithm Context. To our knowledge, this is the first result  comparing this two context tree selection methods. 

Without restrictions on the (possibly infinite) context tree $T_0$, we prove that both methods provide estimators that are with high probability sub-trees of $T_0$ (i.e., a node that is not in $T_0$ does not appear in $\hat{T}$). These bounds are more precise and do not require the conditions assumed in \cite{galves2006,galves2008,leonardi2009}. for this purpose, we derive \emph{``self-normalized''} non-asymptotic deviation inequalities, using martingale techniques inspired from proofs of the Law of the Iterated Logarithm \cite{neveu72,csiszar2002}. %cite le papier de 2004 ?
These inequalities prove interesting in other fields, as for instance in reinforcement learning~\cite{garivier2008,filippiCappeGarivier10KLUCRL}.
On the other hand, we derive upper bounds on the probability of under-estimation by assuming classical mixing conditions on the process generating the sample: with high probability, $\hat{T}$ contains every node of $T_0$ at moderate
height.  This result is based on exponential inequalities derived for a wider class of processes than in \cite{galves2006,galves2008,leonardi2009}.  

Our upper bounds on the probability of over- and under-estimation imply strong consistency of the PML estimators for a larger class of penalizing functions than in \cite{leonardi2009}. Similarly, in the case of the algorithm Context the strong consistency can also be derived for suitable threshold parameters, generalizing the convergence in probability for this estimator obtained previously in \cite{duarte2006}. 

The paper is organized as follows. In Section \ref{sec:notations} we set notation and definitions, we describe in detail the algorithms and we state our main results.
The proof of these results is given in Section \ref{sec:proofs}.
In Section \ref{sec:ccl} we briefly discuss our results. Appendix A contains the statement and proof of the self-normalized deviation inequalities and  Appendix B is devoted to the presentation of exponential inequalities for weak dependent processes.

%%%%%%%%%%%%%%%%%%%%%%%%%%
\section{Notations and results}
\label{sec:notations}
%%%%%%%%%%%%%%%%%%%%%%%%%%
In what follows, $A$ is a finite alphabet; its size is denoted by $|A|$.
$A^j$ denotes the set of all sequences of length $j$ over $A$, in particular $A^0$ has only one element, the empty sequence. 
 We denote by $A^* = \bigcup_{k\geq 0} A^k$ the set of all finite sequences on alphabet $A$
and $A^\infty$ will denote the set of all semi-infinite sequences $v = (\dotsc, v_{-2},v_{-1})$ of symbols in $A$. 
The length of the sequence $w\in A^*$ is $|w|$. 
For $1\leq i\leq j\leq |w|$, we denote $w_i^j = (w_i,\dots, w_j)\in A^{j-i+1}$ and $v_{-\infty}^{-1}$ denotes the
semi-infinite sequence $(\dotsc, v_{-2},v_{-1}) \in A^\infty$. 
Given  $v\in A^*\cup A^\infty$ and  $w\in A^*$, we denote by $vw$ the sequence obtained by
concatenating the two sequences $v$ and $w$.  
We say that the sequence $s\in A^*$ is a \emph{suffix} of the sequence $w\in A^*\cup A^\infty$ if
there exists a sequence $u\in A^*\cup A^\infty$ such that $w = us$.
In this case we write 
$w \succeq s$ or $s\preceq w$. When $|u|\geq 1$ we say that $s$ is a \emph{proper} suffix of $w$ and we write 
$ w \succ s$ or $s\prec w$.  

A set  $T\subset A^*\cup A^\infty$ is a \emph{tree} if no sequence
$s \in T$ is a proper suffix of another sequence $w \in T$.
The \emph{height} of the tree $ T$ is defined as
\begin{equation*}
h( T) = \sup\{|w| : w\in T\}. 
\end{equation*}
If $h( T)<+\infty$ 
we say that $ T$ is \emph{bounded}
and we denote by $|T|$ the cardinality of $T$. 
If $h(T)=+\infty$ 
we say that  $ T$ is
\emph{unbounded}. 
The elements of $T$ are also called the \emph{leaves} of $T$. An \emph{internal node} of $T$ is a proper suffix of a leaf.
For any sequence $w\in A^*\cup A^\infty$ and for any tree $T$,
we define the tree $T_w$ as the set of leaves in $T$ which have $w$
as a suffix, that is
\[
T_w=\{u\in T\colon u \succeq w\}.
\]
Given a tree $ T$ and an integer $K$ we will denote by $ T|_K$ the tree $ T$ \emph{truncated} to level $K$, that is
 \begin{equation*}
  T|_K = \{w \in  T\colon |w| \le K\} \cup \{ w\in A^K\colon  w \prec u \text{ for some } u\in T\}.
 \end{equation*}
Given two trees $T_1$ and $T_2$ we say that $T_1$ is \emph{included} in $T_2$ (denoted by $T_1 \preceq T_2$ or $T_2 \succeq T_1$) if for any sequence $w\in T_1$ there exists a sequence $u\in T_2$ such that $w \preceq u$; in other words, all leaves of $T_1$ are either leaves or internal nodes of $T_2$.

Consider a stationary ergodic stochastic process $\{X_t: t\in\Z\}$
over $A$. Given a sequence $w\in A^*$ we denote by 
\begin{equation*}
 p(w) \,=\,  \P(X_1^{|w|} = w)
\end{equation*}
the stationary probability of the cylinder defined by the sequence $w$.
If $p(w) > 0$ we write
\begin{equation*}
p(a|w) \,=\, \P ( X_0 =a \given X_{-|w|}^{-1}=w)\,.
\end{equation*}

\begin{definition}\label{context}
A sequence $w\in A^*$ is a \emph{finite context} for the process $\{X_t\colon t\in\Z\}$ if it satisfies 
\begin{enumerate}
\item $p(w)>0$;
\item for any sequence $v\in A^*$ such that $p(v)>0$ and $v\succeq w$,
\begin{equation*}
\P ( X_0 =a \given X_{-|v|}^{-1}=v) \,=\, p(a|w),
\quad\text{for all } a\in A;
\end{equation*}
\item no proper suffix of $w$ satisfies 1. and 2.
\end{enumerate}
An \emph{infinite context} is a semi-infinite
sequence $w_{-\infty}^{-1}\in A^\infty$ such that any of its finite suffixes $w_{-j}^{-1}$, $j=1,2,\dotsc$  is a context. In what follows the term \emph{context}
will refer to a finite or infinite context.
\end{definition}

It can be seen that the set of all contexts of the process $\{X_t\colon t\in\Z\}$ is a tree. 
This is called the \emph{context tree} of the process. 
For example, the context tree of an i.i.d. process is $A^0$ and the context tree of a generic Markov chain of order $1$ is $A^1 = A$.
In what follows, we will denote by $T_0$ the context tree of the process $\{X_t\colon t\in\Z\}$.

Let $d\leq n$ be positive integers. 
Let  $X_{-d+1},\dots, X_0, X_{1}, \dots X_n$ be a sequence distributed according to  $\P$.
For any sequence $w\in A^*$ and any symbol $a\in A$ we denote by $N_n(w, a)$ the
number of occurrences of symbol $a$ in $X_1^n$ that are preceded by an occurrence of $w$,
 that is:
\begin{equation} \label{eq:Nn}
N_n(w,a)=\sum_{t=1}^{n}\1\{X_{t-|w|}^{t-1}=w,X_t = a\}. 
\end{equation}
The sum $\sum_{a \in A}N_n(w,a)$ is denoted by $N_n(w)$. 
%We emphasize that $N_n(w,a)$ may differ from $N_n(wa)$ by at most $1$, if $X_{n-|w|}^n = wa$; this remark will simplify the discussion in the sequel.

We will denote by $\V_n$ the set of all sequences $w\in A^*$ that appear at least once in the sample, that is
\[
\V_n = \{w\in A^*\colon N_n(w)\geq 1\}\,.
\]
\begin{definition}\label{accept}
We will say that a tree $T\subset \V_n$ is \emph{acceptable} if it satisfies the following 
conditions: 
\begin{enumerate}
 \item $h(T)\leq d$; \text{ and }
 \item every sequence $w\in A^*$ such that $N_n(w)\geq 1$ belongs to $T$ or has a proper suffix that belongs to  $T$.
 \end{enumerate}
 \end{definition}
Then, our set of candidate trees, denoted by $\sft_n$, will be the set of all acceptable trees. 
Our goal is to select a tree $T\in\sft_n$ as close as possible to $T_0$, in some sense that will be formally given below.
Note that $d$ may depend on $n$, so that the set of candidate trees is allowed to grow with the sample size. The symbols $X_{-d+1},\dots, X_0$ are only observed to ensure that, for every candidate tree $T$, the context of $X_i$ in $T$ is well defined, for every $i=1,\dotsc, n$.
Alternatively, if $X_{-d+1},\dots, X_0$ were not assumed observed, similar results would be obtained by using quasi-maximum likelihood estimators~\cite{galves:garivier:gassiat:2010}. 
Given a tree $ T\subset \cup_{j=1}^d A^j$, the maximum likelihood of the sequence
$X_1,\dotsc,X_n$ is given by
\begin{equation}\label{pmle}
\hat\P_{\ML, T}(X_1^n) =  \prod_{w\in  T}\prod_{a\in A} 
\hat p_n(a|w)^{N_n(w,a)},
\end{equation}
where the empirical probabilities $\hat p_n(a|w)$ are 
\begin{equation}\label{tp}
\hat p_n(a|w) = \frac{N_n(w,a)}{N_n(w)}
\end{equation}
if $N_n(w)>0$ and $\hat p_n(a|w)= 1/|A|$ otherwise. 
For any sequence $w\in A^*$ we define
\begin{equation*}
\hat\P_{\ML,w}(X_1^n) =  \prod_{a\in A} \hat p_n(a|w)^{N_n(w,a)}.
\end{equation*}
Hence, we have 
\begin{equation*}
\hat\P_{\ML, T}(X_1^n) =  \prod_{w\in  T}   \hat\P_{\ML,w}(x_1^n).
\end{equation*}
In order to measure discrepancy between two probability measures over $A$ we use  the \emph{K\"ullback-Leibler divergence}, defined for two probability measures $P$ and $Q$ on $A$ by
$$D(P;Q) = \sum_{a\in A} P(a)\log\frac{P(a)}{Q(a)}$$
where, by convention, $P(a) \log\frac{P(a)}{Q(a)}=0$ if $P(a)=0$ and $P(a) \log\frac{P(a)}{Q(a)}=+\infty$ if $P(a) > Q(a)=0$. 

%%%%%%%%%%%%%%%%
\subsection{The algorithm Context}
%%%%%%%%%%%%%%%%

The algorithm Context introduced by J. Rissanen in \cite{rissanen1983} computes, for each node of a given tree, a discrepancy measure between the transition probability associated to this context and the corresponding transition probabilities of the nodes obtained by concatenating a single symbol to the context. Beginning with the largest leaves of a candidate tree,  if the discrepancy measure is greater than a given threshold, the contexts are maintained in the tree; otherwise, they are pruned.  The procedure continues until no more pruning of the tree can be performed.

For all sequences  $w\in \V_n$ let 
\[\Delta_n(w) = \sum_{b\colon  bw\in \V_n} N_{n}(bw)D\left( \hat{p}_n(\cdot |bw) ;  \hat{p}_n(\cdot |w) \right).\]

\begin{rem}
We use  here the original choice of divergence $\Delta_n(w)$ proposed  by J. Rissanen in \cite{rissanen1983}, but other  possibilities have been proposed in the literature (see for instance \cite{buhlmann1999,galves2006}).
\end{rem}

We will denote the threshold used in algorithm Context on samples of length $n$ by $\delta_n$, where $(\delta_n)_{n\in\N}$ is a sequence of positive real numbers such that $\delta_n\to+\infty$ and $\delta_n/n \to 0$ when $n\to+\infty$. 
For a sequence $X_{1}^n$, let  $\C_w(X_1^n)\in\{0,1\}$ be 
an indicator function  defined for all  $w\in \V_n$  by the following induction:
\begin{equation}\label{funct_cont}
\C_w(X_1^n) = 
\begin{cases}
 0, & \hbox{ if } N_{n}(w) \leq 1 \text{ or } |w|\geq d,\\
  \max\{ \1\{ \Delta_n(w) \geq \delta_n \} , \max_{b\in A} \C_{bw}(X_1^n)\}, &\hbox{ if } N_{n}(w) > 1\text{ and }  |w|< d.
\end{cases}
\end{equation}
With these definitions, the context tree estimator $\hat{T}_C(X_1^n)$ is the set given by 
\begin{equation}\label{alg_cont}
\hat{T}_C(X_1^n) = \{w\in \V_n\colon \C_w(X_1^n)=0 \text{ and } \C_u(X_1^n)=1 \text{ for all } u\prec w \}
\end{equation}

\subsection{The penalized maximum likelihood criterion}

The penalized maximum likelihood criterion for the sequence $X_1^n$ is defined by
\begin{equation}\label{tpml}
\hat{T}_{PML}(X_1^n) =  \argmax_{T\in\sft_n} \left\{\log  \hat\P_{\ML, T}(X_1^n) - |T|f(n) \right\},
\end{equation}
where $f(n)$ is some positive function such that $f(n)\to+\infty$ and $f(n)/n\to 0$ when $n\to\infty$.

This class of context tree estimators was first considered by Csisz\'ar and Talata in \cite{csiszar2006}, who introduced the Bayesian Information Criterion (BIC) for context trees and proved its consistency. The BIC leads to the choice of the penalty function  $f(n) = (|A|-1)\log(n)/2$.  
It may first appear practically impossible to compute $\hat{T}_{PML}(X_1^n)$, because the maximization in (\ref{tpml}) must be performed over the set of all candidate trees. Fortunately, Csisz\'ar and Talata showed in their article \cite{csiszar2006} how to adapt the Context Tree Maximizing (CTM) method \cite{willems1995} in order to obtain a simple and efficient algorithm computing $\hat{T}_{PML}(X_1^n)$. 
As the representation of the estimator $\hat{T}_{PML}(X_1^n)$ given by this algorithm  is important for the proof of our results, we briefly present it here. Define recursively, for any  $w\in \V_n$, with $|w|<d$, the value
\begin{equation}\label{value}
  V_w(X_1^n) =    \max\Big\{e^{-f(n)} \hat\P_{\ML,w}(X_1^n),\; \prod_{b\in A\colon bw \in \V_n}V_{bw}(X_1^n)\Big\}
   \end{equation}
and the indicator
\begin{equation}\label{ind}
  \XX{w}{}=   \1\Big\{ \prod_{b\in A\colon bw \in \V_n}V_{bw}(X_1^n) > e^{-f(n)} \hat\P_{\ML,w}(X_1^n)\Big\}\,.
\end{equation}
By convention, if
 $\{b\in A\colon bw \in \V_n\} = \emptyset$ or if
 $|w|=d$ then $ V_w(X_1^n) = e^{-f(n)} \hat\P_{\ML,w}(X_1^n)$ and $  \XX{w}{}=0$. As shown in \cite{csiszar2006}, it holds that

\begin{equation}\label{alg_pml}
\hat{T}_{PML}(X_1^n) = \{w\in \V_n\colon \XX{w}{}=0 \text{ and } \XX{u}{}=1 \text{ for all } u\prec w \}\;.
\end{equation}

\subsection{Results}

In this subsection we present the main results of this article.
First, we show that the empirical tree given by the algorithm Context is always included in the tree given by the penalized 
maximum likelihood estimator, if the threshold $\delta_n$ is smaller than the penalization  function $f(n)$. 

\begin{prop}
\label{prop:comparaison}
For any $n\geq 1$ and all sequences $X_1^n$, if $\delta_n\leq f(n)$ then 
\[\hat{T}_{PML}(X_1^n) \;\preceq \;\hat{T}_{C}(X_1^n)\,.\]
\end{prop}

In the sequel we will assume that the cutoff sequence of the algorithm Context equals the penalization term of the penalized maximum likelihood
estimator, in order to allow a  unified treatment of the two algorithms. That is, we will assume that  $\delta_n = f(n)$ for any $n\geq 1$. 

We now state a new bound on the probability of over-estimation that does not require any mixing hypotheses on the underlying process.
\begin{theorem}
\label{th:overestimation}
For every  $n\geq 1$ it holds that
\begin{equation}\label{factor}
\P\left( \hat{T}(X_1^n) \preceq T_0  \right)\; \geq\; 1-e\left(\delta_n\log(n) +|A|^2\right)n^2\exp\left(-\frac{\delta_n}{|A|^2}\right)\,, 
\end{equation}
where $ \hat{T}(X_1^n)=\hat{T}_{PML}(X_1^n)$ or $\hat{T}(X_1^n)=\hat{T}_{C}(X_1^n)$.
\end{theorem}

\begin{rem}
Theorem~\ref{th:overestimation} is  proven without assuming any bound on the height of the hypothetical trees. That is, the result remains valid even if $d=-\infty$. But if the candidate trees have only a limited number of nodes, possibly depending on $n$  (see, e.g, \cite{rissanen1983,csiszar2006}), a straightforward modification of the proof shows that 
\[
\P\left( \hat{T}(X_1^n) \preceq T_0  \right) \;\geq\; 1-2e\left(\delta_n\log(n) +|A|^2\right)k(n)\exp\left(-\frac{\delta_n}{|A|^2}\right)\,,
\]
where $k(n)$ is the maximal number of nodes of a candidate tree.
In particular, if the height of the trees is smaller than a function $d(n)$ (possibly constant) then $k(n) = |A|^{d(n)}$.
\end{rem}

The problem of under-estimation in context tree models is very different, and requires additional hypotheses on the process 
$\{X_t\colon t\in\Z\}$. 
For any $w\in A^*$ with $p(w)>0$ define the coefficient
\[
\beta(w,r) = \max_{u\in A^r}\max_{a\in A}\, \{|p(a|w)-p(a|uw)|\}\,.
\]
The \emph{continuity rate} of the process $\{X_t\colon t\in\Z\}$ is the sequence $\{\beta_k\}_{k\in \N}$ where  
\[
\beta_k = \max_{w\in A^k}\sup_{r\geq 1}\, \{\beta(w,r)\}\,.
\]
Define also the \emph{non-nullness}  coefficient
\begin{equation}\label{alpha0}
\alpha_0 :=  \sum_{a\in A} \inf_{w\in T_0} \{\, p(a|w)\}\,.
\end{equation}
Our underestimation error bounds will rely on the following assumption.
\begin{assumption}\label{ass:phi}
The process $\{X_t\colon t\in\Z\}$ satisfies the following conditions
\begin{enumerate}
\item $\alpha_0>0$ (weakly non-nullness) and
\item $\beta := \sum_{k\in\N} \beta_k \;< \;+\infty$\, (summable continuity rate). 
\end{enumerate}
\end{assumption}

These are classical hypotheses for processes of infinite memory, which are also referred to as \emph{chains of type A}, 
see for instance \cite{fernandez2002} and references therein. 
 
To establish upper bounds for the probability of under-estimation we will consider
the truncated tree $T_0|_K$, for any given constant $K\in\N$. Note that in the case $T_0$ is a finite tree, 
$T_0|_K$ coincides with $T_0$ for a sufficiently large constant $K$. 
The bounds are stated in the following theorem.
 
\begin{theorem}
\label{th:underestimation}
Assume the process $\{X_t\colon t\in\Z\}$ satisfies Assumption~\ref{ass:phi}. Let $K\in\N$ and let $d$ be such that 
\begin{equation}\label{eq:hypcont}
\min_{w\prec u\in T_0|_K} \max_{r\leq d-|w|}\; \{ \beta(w,r) \}\geq \epsilon >0\,. %\tag{$\star$}
\end{equation}
Then, there exists $n_0\in\N$ such that for any $n\geq n_0$ it holds that
\begin{equation}\label{upp}
\P\left(  T_0|_K \preceq \;  \hat{T}(X_1^n)|_K  \right)\; \geq \;1\,-\,
3e^{\alpha_0/32e^2|A|^2(|A|\beta+2\alpha_0)} |A|^{2+K} \exp\Bigl[\frac{-n\epsilon^2\,[p_{\min}^d-\frac{8|A|d\,f(n)}{\epsilon^2n}]^2}
{16(d+1)}\Bigr]\,,
\end{equation}
where $\hat{T}(X_1^n)=\hat{T}_{PML}(X_1^n)$ or $\hat{T}(X_1^n)=\hat{T}_{C}(X_1^n)$ and  $p_{\min}=\min_{a\in A,w\in A^d}\{p(a|w)\colon p(a|w)>0\}$ .
\end{theorem}

\begin{rem}
It can be seen that for any $K\in\N$ there is always a value of $d$ such that \eqref{eq:hypcont} holds. This hypothesis can be avoided by letting $d$ increase with the sample size $n$ and by controlling the upper bounds in (\ref{upp}). 
Extensions of Theorem~\ref{th:underestimation} can also be obtained by allowing $K$ to be a function of the sample size $n$. In this case, the rate at which $K$ increases must be controlled together with the rate at which  $\epsilon$ and $p_{\min}$ decrease with the sample size. This leads to a rather technical condition, see for instance \cite{talata2009}.
\end{rem}

Finally, the next theorem states the strong consistency of the estimators $\hat T_C(X_1^n)$ and $\hat T_{PML}(X_1^n)$ for appropriate  threshold parameters and penalizing functions, respectively.

\begin{theorem}\label{strong:cons}
Assume the hypotheses of Theorem~\ref{th:underestimation} are met. Then for any threshold parameter $(\delta_n)_{n\in\N}$ such that
\[
\sum_{n\in\N}\exp \left(-\frac{\delta_n}{|A|^2} + \log(\delta_n\log(n))\right) \;<\;+\infty
\]
we have $\hat T_C(X_1^n)|_K = T_0|_K $ eventually almost surely as $n\to+\infty$. Similarly, if we choose $f(n)=\delta_n$ we have
$\hat T_{PML}(X_1^n)|_K = T_0|_K $ eventually almost surely as $n\to+\infty$.
\end{theorem}

%%%%%%%%%%%%%%%%%%%%%%%%%%
\section{Proofs} \label{sec:proofs}
%%%%%%%%%%%%%%%%%%%%%%%%%%
\subsection{Proof of Proposition~\ref{prop:comparaison}.}
\label{sec:comp}

We must prove that a leaf in  $\hat{T}_{PML}(X_1^n)$ is always a leaf or an internal node in $\hat{T}_{C}(X_1^n)$. 
By the characterization of $\hat T_{C}(X_1^n)$ and $\hat T_{PML}(X_1^n)$ given by equations (\ref{alg_cont}) and (\ref{alg_pml}), respectively, this is equivalent to proving that $\XX{w}{}\leq \C_w(X_1^n)$ for all $w\in \V_n$ with $|w| < d$. In fact, assume that $\XX{w}{}=1$ implies $\C_w(X_1^n)=1$, and take $w\in\hat T_{PML}(X_1^n)$; then, either $|w|=d$ and $w\in \hat T_{C}(X_1^n)$, or it holds that for all $u\prec w$, $\XX{u}{}=1$, which implies by assumption that $\C_u(X_1^n)=1$. Now, if $\C_w(X_1^n) = 0$, then $w \in \hat T_{C}(X_1^n)$; otherwise, $w$ is a proper suffix of a sequence $v\in T_C(X_1^n)$. In any case, $w$ is a leaf or an internal node of $\hat{T}_C(X_1^n)$.

 Assume  there exists $w\in \V_n$, $|w|<d$,  such that 
$\XX{w}{}=1$ and $\C_w(X_1^n)=0$.  Note that by (\ref{funct_cont}), $\C_w(X_1^n)=0$ implies $\C_{uw}(X_1^n)=0$
for all $uw\in \V_n$, $|uw|\leq d$; hence, $w$ can be chosen such that $\XX{bw}{}=0$ for any $bw\in \V_n$, $b\in A$.
In this case we have, by  the definitions  (\ref{value})   and (\ref{ind}) that 
\begin{align}\label{eq:xxleqc}
e^{-f(n)} \hat\P_{\ML,w}(X_1^n)\; &< \; \prod_{b\colon  bw\in \V_n} V_{bw}(X_1^n) \\
& = \; \prod_{b\colon  bw\in \V_n} e^{-f(n)} \hat\P_{\ML,bw}(X_1^n) \,.
\end{align}
 The equality in the second line of the last expression follows by the fact that $\XX{bw}{}=0$ for any $bw\in \V_n$, $b\in A$; therefore 
we must have $V_{bw}(X_1^n) = e^{-f(n)} \hat\P_{\ML,bw}(X_1^n)$ for any $bw\in \V_n$, $b\in A$. 

Now, observe that for any $a\in A$, $N_n(w,a) = \sum_{b\colon  bw\in \V_n} N_n(bw,a)$
and $|\{b\colon  bw\in \V_n\}|\geq 2$.  If  not, $N_n(w,a)$ would be equal to $N_n(cw,a)$ for some $c\in A$ and for all $a\in A$, implying that $\hat\P_{\ML,cw}(X_1^n) = \hat\P_{\ML,w}(X_1^n)$; hence  
\[\prod_{b\colon  bw\in \V_n} V_{bw}(X_1^n) =  V_{cw}(X_1^n)  =  e^{-f(n)} \hat\P_{\ML,cw}(X_1^n)  =  e^{-f(n)} \hat\P_{\ML,w}(X_1^n) \] 
and thus, by definition, $\XX{w}{}=0$. Using these facts, and taking logarithm on both sides of Inequality~\eqref{eq:xxleqc}, we obtain
\begin{align*}
 \Big(\big|\{b\colon  bw\in \V_n\}\big|-1\Big)\,f(n)\; &< \; \sum_{b\colon  bw\in \V_n} \sum_{a\in A} N_n(bw,a) \log \frac{\hat p_n(a|bw)}{\hat p_n(a|w)} \\
 & = \;  \sum_{b\colon  bw\in \V_n} N_{n}(bw)\,D\left( \hat{p}_n(\cdot |bw) ;  \hat{p}_n(\cdot |w) \right)\; = \; \Delta_n(w)\,.
\end{align*}
Therefore, if $\delta_n \leq f(n)$ we have
$\delta_n < \Delta_n(w)$ which contradicts the fact that $\C_w(X_1^n)=0$. This concludes the proof of Proposition~\ref{prop:comparaison}.

%%%%%%%%%%%%%%%%%%%%%%%%%%
\subsection{Proof of Theorem~\ref{th:overestimation}.}
\label{sec:over}
%%%%%%%%%%%%%%%%%%%%%%%%%%

We will prove the result for the case $ \hat{T}(X_1^n)=\hat{T}_{C}(X_1^n)$. The case $ \hat{T}(X_1^n)=\hat{T}_{PML}(X_1^n)$
follows straightforwardly from Proposition~\ref{prop:comparaison} and equality $f(n)=\delta_n$.  

Let $O_n$ be the event $\bigl\{\hat{T}_{C}(X_1^n) \npreceq T_0\bigr\}$.
Overestimation occurs if at least one internal node $w$ of $\hat{T}_{C}(X_1^n)$ has a (non necessarily proper) suffix $s$ in $T_0$; that is, if there exists a (possibly empty) sequence $u$ such that $w=us$. Thus,  with a little abuse of notation $O_n$ can be written as
\[
O_n= \bigcup_{s\in T_0} \bigcup_{u\in A^*} \left\{\Delta_n(us) > \delta_n\right\}.
\]
For any sequence $w\in A^*$ we have that $\hat p_n(\cdot|w)$ are the maximum likelihood estimators of the transition probabilities $p(\cdot|w)$, therefore we have that 
\begin{align*}
\Delta_n(w) &= \sum_{b\in A} N_{n}(bw)D\left( \hat{p}_n(\cdot |bw)  ;  \hat{p}_n(\cdot |w) \right)\\
 &= \sum_{b\in A} N_{n}(bw) \sum_{a\in A} \left(\hat{p}(a|bw)\log \hat{p}(a|bw) - \hat{p}(a|bw)\log \hat{p}(a|w) \right)\\
 &= \left( \sum_{b\in A} N_{n}(bw) \sum_{a\in A} \hat{p}(a|bw)\log \hat{p}(a|bw) \right) -  \sum_{b\in A} \sum_{a\in A} N_{n}(bw,a)\log \hat{p}(a|w)\\
&= \left( \sum_{b\in A} N_{n}(bw) \sum_{a\in A} \hat{p}(a|bw)\log \hat{p}(a|bw) \right) -  \sum_{a\in A} N_{n}(w,a)\log \hat{p}(a|w)\\
&\leq \left( \sum_{b\in A} N_{n}(bw) \sum_{a\in A} \hat{p}(a|bw)\log \hat{p}(a|bw) \right) -  \sum_{a\in A} N_{n}(w,a)\log p(a|w)\\
&= \left( \sum_{b\in A} N_{n}(bw) \sum_{a\in A} \hat{p}(a|bw)\log \hat{p}(a|bw) \right) -  \sum_{b\in A} \sum_{a\in A} N_{n}(bw,a)\log p(a|w)\\
& = \sum_{b\in A} N_{n}(bw) \sum_{a\in A} \left(\hat{p}(a|bw)\log \hat{p}(a|bw) - \hat{p}(a|bw)\log p(a|w) \right)\\
& = \sum_{b\in A} N_{n}(bw)D\left( \hat{p}_n(\cdot |bw)  ;  p(\cdot |w) \right)
\end{align*}
Hence, as for all $b\in A$ it holds that $p(\cdot | w) = p(\cdot | bw)$ we obtain  
\[
\P\left(\Delta_n(w)>\delta_n\right) \,\leq \,\P\Biggl(\,  \sum_{b\in A} N_{n}(bw)D\left( \hat{p}_n(\cdot |bw)  ;  p(\cdot |bw) \right) >  \delta_n\Biggr).
\]
Using Theorem \ref{th:borneD}, stated in Appendix~A, it follows that 
\begin{align*}
\P(O_n) &\leq \sum_{s\in T_0} \sum_{u\in A^*} \P\left(\Delta_n(us) > \delta_n \right) \\
& \leq\sum_{s\in T_0} \sum_{u\in A^*} \P\left(\,\sum_{b\in A } N_{n}(bus)D\left( \hat{p}_n(\cdot |bus)  ;  p(\cdot |bus)\right) > \delta_n\right)\\
& \leq \sum_{s\in T_0} \sum_{u\in A^*} \sum_{b\in A}\,\P\left(N_{n}(bus)D\left( \hat{p}_n(\cdot |bus)  ;  p(\cdot |bus)\right)  > \frac{\delta_n}{|A|}\, \Big|\, N_{n}(bus)>0\right) \P\left( N_{n}(bus)>0\right)\\
& \leq 2e\left(\delta_n\log n +|A|^2\right)\exp\left(-\frac{\delta_n}{|A|^2}\right) \sum_{s\in T_0} \sum_{u\in A^*}\sum_{b\in A}  \P\left( N_{n}(bus)>0\right)\\
&\leq 2e\left(\delta_n\log n +|A|^2\right)\exp\left(-\frac{\delta_n}{|A|^2}\right) \E[C_n], 
\end{align*}
where $C_n$ denotes the number of different contexts of the symbols in $X_1^n$.
But $C_n$ is always upper-bounded by the number $n(n-1)/2$ of (non-necessarily distinct) contexts of $X_1^n$, and the result follows.

%%%%%%%%%%%%%%%%%%%%%%%%%%
\subsection{Proof of Theorem~\ref{th:underestimation}.}
\label{sec:under}
%%%%%%%%%%%%%%%%%%%%%%%%%%

In this case we will prove the result for the case $ \hat{T}(X_1^n)=\hat{T}_{PML}(X_1^n)$. The case $ \hat{T}(X_1^n)=\hat{T}_{C}(X_1^n)$ follows again from Proposition~\ref{prop:comparaison} and  the assumption that $\delta_n=f(n)$.  

 If $U_n$ denotes the event  $\{T_0|_{K} \npreceq \hat{T}_{PML}(X_1^n)|_K\}$ then 
\begin{equation*}
U_n \;\subset \bigcup_{w\prec u\in T_0|_{K}} \{\,\XX{w}{} = 0\,\} \,.
\end{equation*}
Let ${w\prec u\in T_0|_{K}}$. 
Then we have
\begin{align}
\P\left(\, \XX{w}{} = 0 \,\right)  \;= \;\,\P\left(\;\prod_{a\in A\colon aw\in \V_n} 
V_{aw}(X_1^n)\, \leq\, e^{-f(n)}\hat\P_{\ML,w}(X_1^n)\,\right)\,.
\end{align}
By hypothesis, there exists $r\leq d-|w|$ and $s\in A^r$ such that 
\[
\max_{a\in A} |p(a|w) - p(a|sw)|\,\geq\, \epsilon\,.
\]
If $s=(s_1\dotsc s_r)$, denote by $A_i=A\setminus\{s_i\}$ and let $T$ be the tree given by
\[
T = \cup_{i=2}^r\cup_{b\in A_i}\{bs_i^rw\}\cup\{sw\} \,.
\]
By  definition, for any $aw\in \V_n$ it can be shown recursively that 
\[
V_{aw}(X_1^n) = \max_{T'\in\sft_n} \prod_{v\in T'_{aw}}e^{-f(n)}\hat\P_{\ML,v}(X_1^n) 
\]
see for example Lemma~4.4 in \cite{csiszar2006}.
Therefore,
\begin{align}\label{prob1}
\P\Bigg(\prod_{a\in A\colon aw\in \V_n} &
V_{aw}(X_1^n)\, \leq\, e^{-f(n)}\hat\P_{\ML,w}(X_1^n) \Bigg)\notag \\
& \leq\; \P\left(\; \prod_{u\in T}e^{-f(n)}\hat\P_{\ML,u}(X_1^n)
\, \leq\, e^{-f(n)}\hat\P_{\ML,w}(X_1^n)\;\right)
\end{align}
by noticing that  
\begin{align*}
\prod_{a\in A\colon aw\in \V_n} \max_{T'\in\sft_n} \prod_{v\in T'_{aw}}e^{-f(n)}\hat\P_{\ML,v}(X_1^n) 
\;&\geq\; \prod_{a\in A\colon aw\in \V_n} \prod_{v\in T_{aw}}e^{-f(n)}\hat\P_{\ML,v}(X_1^n)\\
&\geq\; \prod_{u\in T}e^{-f(n)}\hat\P_{\ML,u}(X_1^n)\,.
\end{align*}
Applying logarithm and using that $N_n(w,a) = \sum_{u\in T}N_n(u,a)$ for any $a\in A$ we can write the probability in (\ref{prob1}) by  
\begin{align}\label{prob2}
\P\Bigl(\,\sum_{u\in T} N_n(u) &D(\hat p_n(\cdot|u)\,;\,\hat p_n(\cdot|w)) \,\leq\, (|T| - 1) f(n) \,\Bigr)\notag\\
&\leq\; \P\Bigl(\, N_n(sw) D(\hat p_n(\cdot|sw)\,;\,\hat p_n(\cdot|w)) \,\leq\, (|T| - 1) f(n) \,\Bigr)\,.
\end{align}
Define the events $A_n^{s,w}$ and $B_n^{s,w}$ by 
\begin{align*}
A_n^{s,w} &=\left\{X_1^n\colon N_n(sw) D(\hat p_n(\cdot|sw)\,;\,\hat p_n(\cdot|w))\,\leq\, (|T| - 1) f(n)\right\}\\
B_n^{s,w} &= \left\{X_1^n\colon   D(\hat p_n(\cdot|sw)\,;\,\hat p_n(\cdot|w)) > \epsilon^2/8 \right\}\,.
 \end{align*}
Then we can bound above the probability in (\ref{prob2}) by 
$\P\bigl(\, A_n^{s,w} \cap B_n^{s,w} \,\bigr) + \P\bigl(\, (B_n^{s,w})^c\,\bigr)$. 
To bound the first term note that by Lemma~\ref{lemma:Nn}, if $n$ satisfies
\[
\frac{f(n)}{n}\,<\,\frac{\epsilon^2p(sw)}{8(|T|-1)}
\]
then, using the bound $|T|-1\leq |A|r\leq |A|d$ we obtain 
\begin{align*}
\P\bigl(\, A_n^{s,w} \cap B_n^{s,w} \,\bigr) \;&\leq\; \P\Bigl(\, N_n(sw) \,\leq\, \frac{8(|T| - 1) f(n)}{\epsilon^2} \,\Bigr)\\
& \leq\;\,e^{\alpha_0/8e^2|A|^2(|A|\beta+2\alpha_0)} |A|
\exp \Bigl(\frac{- n\, [p(sw)-\frac{8|A|d\,f(n)}{\epsilon^2n}]^2}{|sw|+1}\Bigr)\,.
\end{align*}
On the other hand, by Lemma~\ref{lemma:div} we have
\[
\P\left(\,(B_n^{s,w})^c \,\right) \;\leq\; 2 e^{\alpha_0/32e^2|A|^2(|A|\beta+2\alpha_0)} (|A|+1)\exp\Bigl[ -n\frac{\epsilon^2p(sw)^2}{16(|sw|+1)}  \Bigr]\,.
\]
We conclude the proof of Theorem~\ref{th:underestimation} by observing that we only 
have a finite number 
of sequences $w\prec u\in T_0|_K$, therefore we obtain 
\begin{align*}
\P(U_n)\;\leq\; 3e^{\alpha_0/32e^2|A|^2(|A|\beta+2\alpha_0)} |A|^{2+K} \exp\left(\frac{- \, n\epsilon^2[p_{\min}^d-\frac{8|A|d\,f(n)}{\epsilon^2n}]^2}
{16(d+1)}\right)\,.
\end{align*}

%%%%%%%%%%%%%%%%%%%%%%%%%%
\subsection{Proof of Theorem~\ref{strong:cons}.}
\label{sec:cons}
%%%%%%%%%%%%%%%%%%%%%%%%%%

The statement of the Theorem follows straightforward from Theorems \ref{th:overestimation} and \ref{th:underestimation} and the Borel-Cantelli Lemma, by noticing that the upper bounds for 
\[
\P(\hat T_C(X_1^n)|_K \neq T_0|_K) \;\leq\;  \P(\hat T_C(X_1^n)|_K \npreceq T_0|_K) \,+\, \P(T_0|_K \npreceq \hat T_C(X_1^n)|_K  )
\]
are summable in $n$. The same reasoning applies to $\hat T_{PML}(X_1^n)$ when $f(n)=\delta_n$.

%%%%%%%%%%%%%%%%%%%%%%%%%%
\section{Discussion}
\label{sec:ccl}

In this paper we showed a relation between two classical algorithms for context tree selection. We proved that for a proper set of parameters, the Penalized Maximum Likelihood estimator always yields a smaller tree than the tree given by the algorithm Context.  This relation between the empirical context trees allows us to derive, in an unified way, non-asymptotic bounds for
the probability of over- and under-estimation of the context tree generating the sample. The tree may be unbounded, and our results apply to processes that do not necessarily have a finite memory. 

Concerning under-estimation, we assume the process satisfies some conditions that implies exponential inequalities for the empirical probabilities. These inequalities were obtained in \cite{galves2008} under a stronger \emph{non-nullness} assumption; namely, that the transition probabilities were lower bounded by a positive constant. In this paper we show that the results also hold for a larger class of processes. 
It is conjectured that similar results cannot be obtained without assuming any non-nullness nor mixing condition of the process. 

Concerning over-estimation no mixing assumption is necessary for Theorem  \ref{th:overestimation} to hold. This improves on and generalizes the results obtained in \cite{galves2008,leonardi2009}. Our proof is based on deviation inequalities obtained for empirical Kullback-Leibler divergence, instead of $L^p$ norm; it appears  that this pseudo-metric is more intrinsic for binomial distributions (and partially also for multinomial distributions), as the binary Kullback-Leibler divergence is the rate function of a Large Deviations Principle.  Deriving similar inequalities is also possible for other distributions and thus other pseudo-metrics, or by using upper-bounds of the Legendre transform of the distribution, as in \cite{garivier2008}.
These type of inequalities are interesting on their own and prove useful in various settings: other applications of similar bounds may be found in \cite{garivier2008,filippiCappeGarivier10KLUCRL,COLT2011}.

From the point of view of most applications, over- and under-estimation play a different role.
In fact, data-generating processes can often not be assumed to have finite memory: the whole dependence structure cannot be recovered from finitely many observations and under-estimation is unavoidable. 
All what can be expected from the estimator is to highlight evidence of as much dependence structure as possible, while maintaining a limited probability of false discovery. 

Our results imply the strong consistency of the algorithm Context for processes of infinite memory, generalizing the convergence in probability of this estimator previously obtained in \cite{duarte2006}. Likewise, the strong consistency  for the PML estimator is also derived  for a larger class of penalizing functions than in \cite{leonardi2009}.

%%%%%%%%%%%%%%%%%%%%%%%%%%
\section*{Appendix A: Martingale deviation inequalities}
\label{sec:mart}
\renewcommand{\thesection}{A}
%%%%%%%%%%%%%%%%%%%%%%%%%%

This section contains the statement and derivation of two deviation inequalities that are useful to prove the results of
this paper. As they are interesting on their own, we include them in a separate section.
The ingredients of the proofs are mostly inspired by \cite{neveu72}, see also \cite{csiszar2002}. 

We briefly recall some notation so as to keep this section self-contained.
Let $\left( X_n \right)_{n\in\Z}$ be a stationary process whose (possibly infinite) context tree is $T_0$, and let $\F_n$ be the $\sigma$-field generated by  $(X_j)_{j\leq n}$.
For $k\in\N$, $w\in A^k$, denote $p(b|w)=\P\left( X_{k+1}=b | X_{1}^{k}=w \right)$. For $j\geq 1$, define
\[
\x_j = \1\{X_{j-k}^{j-1}=w\}\quad\text{ and }\quad
\X_j = \1\{X_{j-k}^j=wb\}\;,
\]
so that $N_n(w) = \sum_{j=1}^n\x_j$ and $N_n(w,b) = \sum_{j=1}^n\X_j$.
Denote $\ph_n(b |w)= N_n(w,b) / N_{n}(w)$. 
 The Kullback-Leibler divergence between Bernoulli variables will be denoted by $d$: for all $p,q\in[0,1]$,
$$d(p;q) = p\log\frac{p}{q}+(1-p)\log\frac{1-p}{1-q}.$$

\begin{prop}\label{prop:borned} Let $k$ be a positive integer, let $w\in A^k$ and let $b\in A$. 
Then for any $\delta>0$
  $$\P\left[ N_{n}(w)d\left( \ph_n(b|w) ; p(b|w) \right)>\delta\right] \;\leq\; 2e\left\lceil \delta\log(n) \right\rceil \exp(-\delta).$$
\end{prop}

\begin{proof}
Denote by $p=p(b|w)$, $N_n = N_n(w)$, $S_n = N_n(w, b)$  and $\ph_n= S_n/N_{n}$. For  every $\lambda >0$, let 
$$\phi_p(\lambda) = \log \E\left[ \exp\left( \lambda X_1 \right) \right]=\log \left( 1-p+p \exp\left( \lambda \right) \right).$$
Let also $W^{\lambda}_0=1$ and for $t\geq 1$, $$W^{\lambda}_t = \exp(\lambda S_t - N_{t-1} \phi_p(\lambda)).$$
First, note that $\left(W^{\lambda}_t\right)_{t\geq 0}$ is a martingale relative to $\left( \F_t \right)_{t\geq 0}$ with expectation $\E[W^{\lambda}_0] = 1$. In fact, 
\begin{eqnarray*}
\E\left[ \exp\left( \lambda\left( S_{t+1}-S_t \right) \right) | \F_t\right] &=&\E\left[ \exp\left( \lambda \X_{t+1}\right) | \F_t\right]\\
& =& \exp\left( \x_t\phi_p\left( \lambda \right) \right)\\
&=& \exp \left( \left( N_t-N_{t-1} \right) \phi_p\left( \lambda \right) \right)
\end{eqnarray*}
which can be rewritten as
\[
\E\left[\exp\left( \lambda S_{t+1}-N_t\phi_p\left( \lambda \right) \right)  | \F_{t}\right] = \exp\left( \lambda S_{t}-N_{t-1}\phi_p\left( \lambda \right) \right).
\]

To proceed, we make use of the so-called 'peeling trick' \cite{massart2007}: we divide the interval $\{1,\dots, n\}$ of possible values for $N_{n}$ into "slices" $\{t_{k-1}+1,\dots,t_k\}$ of geometrically increasing size, and treat the slices independently.
We may assume that $\delta>1$, since otherwise the bound is trivial. Take $\eta = 1/(\delta-1)$, let $t_0=0$ and for $k\in\N^{*}$, let $t_k = \left\lfloor (1+\eta)^k \right\rfloor$.
Let $m$ be the first integer such that $t_m\geq n$, that is $m=\left\lceil\frac{\log n}{\log 1+\eta}\right\rceil$.
Let $A_k = \left\{t_{k-1} < N_{n}\leq t_k\right\}\cap \left\{ N_{n}d\left( \ph_n ; p \right)>\delta\right\}$. We have:
\begin{equation}
\label{eq:peeling:basic}
\P\left(N_{n}d\left( \ph_n ; p \right)>\delta\right)  \leq \P\left( \bigcup_{k=1}^m A_k\right) \leq \sum_{k=1}^m \P\left(A_k\right).
\end{equation}

We upper-bound the probability of $A_k \cap \{\ph_n>p\}$, the same arguments can easily be transposed for left deviations.
Let $s$ be the smallest integer such that  $\delta/(s+1)\leq d(1;p)$; if $N_n\leq s$, then $N_nd(\ph_n, p)\leq s d(\ph_n, p)\leq sd(1,p)< \delta$ and $\P(N_{n}d(\ph_n, p)\geq \delta, \ph_n>p)=0$.
Thus, $\P\left(A_k\right)=0$ for all $k$ such that $t_k\leq s$.

Take $k$ such that $t_k>s$, and let $\tilde{t}_{k-1} = \max\{t_{k-1}, s\}$.
Let $x \in ]p, 1]$ be such that $d(x;p)=\delta/N_{n}$, and let $\lambda(x)=\log(x\left( 1-p \right))-\log (p\left( 1-x \right))$, so that 
$d(x; p) = \lambda(x)x-\phi_p(\lambda).$
Let $z$ such that $z\geq p$ and $d(z, p)=\delta/(1+\eta)^{k}$.
Observe that:
\begin{itemize}
 \item if $N_{n}>t_{k-1}$, then $$d(z; p) = \frac{\delta}{(1+\eta)^{k}}  \geq \frac{\delta}{(1+\eta)N_{n}};$$
 \item if $N_{n}\leq t_k$ then, as
 $$d(\ph_n; p)>\frac{\delta}{N_{n}}>\frac{\delta}{(1+\eta)^k} = d(z; p),$$
  we have :
  $$\ph_n\geq p \hbox{ and } d(\ph_n; p)>\frac{\delta}{N_{n}}\implies \ph_n\geq z.$$
\end{itemize}
Hence, on the event $\left\{t_{k-1}<N_{n}\leq t_k\right\} \cap \left\{\ph_n\geq p\right\} \cap \left\{d(\ph_n; p)>\frac{\delta}{N_{n}}\right\}$ it holds that 
$$\lambda(z) \ph_n - \phi_p(\lambda(z)) \geq 
\lambda(z) z - \phi_p(\lambda(z)) = d(z; p)
\geq\frac{\delta}{(1+\eta)N_{n}}.$$
Putting everything together,
\begin{align*}
\left\{ \tilde{t}_{k-1}<N_{n}\leq t_k \right\}\cap\left\{ \ph_n\geq p  \right\}\cap\left\{d(\ph_n; p) \geq \frac{\delta}{N_{n}} \right\}&\subset \left\{ \lambda(z) \ph_n - \phi_p(\lambda(z)) \geq  \frac{\delta}{N_{n}\left( 1+\eta \right)}  \right\}\\
& \subset \left\{  \lambda(z)S_n - N_{n} \phi_p(\lambda(z)) \geq \frac{\delta}{1+\eta} \right\} \\
& \subset \left\{  W^{\lambda(z)}_n > \exp \left(\frac{\delta}{1+\eta}\right)\right\}.
\end{align*}
As $\left(W^{\lambda}_t\right)_{t\geq0}$ is a martingale, $\E\left[W^{\lambda(z)}_n\right] = \E\left[W^{\lambda(z)}_0\right] =1$, and the Markov inequality yields:
\begin{align}
\P\left(\left\{\tilde{t}_{k-1}<N_{n}\leq t_k\right\}\cap \left\{ \ph_n\geq p \right\} \cap \left\{N_{n}d(\ph_n, p)\geq \delta \right\} \right) &\leq 
\P\left(W^{\lambda(z)}_n > \exp\left(\frac{\delta}{1+\eta}\right) \right) \label{eq:umarkov} \\
&\leq \exp\left(-\frac{\delta}{1+\eta}\right).\nonumber
\end{align}
Similarly,
\begin{equation*}
\P\left(\left\{\tilde{t}_{k-1}<N_{n}\leq t_k\right\} \cap \left\{ \ph_n\leq p \right\} \cap \left\{N_{n}d(\ph_n, p)\geq \delta \right\}\right)\leq \exp\left(-\frac{\delta}{1+\eta}\right),\end{equation*}
so that
\begin{equation*}
\P\left(\left\{\tilde{t}_{k-1}<N_{n}\leq t_k\right\} \cap \left\{N_{n}d(\ph_n, p)\geq \delta \right\} \right) \leq 2\exp\left(-\frac{\delta}{1+\eta}\right).\end{equation*}
Finally, by Equation \eqref{eq:peeling:basic},
$$ \P\left( \bigcup_{k=1}^m \left\{\tilde{t}_{k-1}<N_{n}\leq t_k\right\} \cap \left\{N_{n}d(\ph_n, p) \geq \delta\right\}\right)  \leq 2m \exp\left(-\frac{\delta}{1+\eta}\right).$$
But as $\eta = 1/(\delta-1)$, $m=\left\lceil\frac{\log n}{\log \left( 1+1/(\delta-1) \right)}\right\rceil$ and as $\log(1+1/(\delta-1))\geq 1/\delta$, we obtain:
 $$\P\left( N_{n}d(\ph_n, p) \geq \delta\right) \leq 2\left\lceil\frac{\log n}{\log \left( 1+\frac{1}{\delta-1}\right)}\right\rceil 
  \exp(-\delta+1)\leq 2e\left\lceil \delta\log(n) \right\rceil \exp(-\delta).
 $$ 
\end{proof}

\begin{rem}\label{rem:bornedcond}
The bound of Proposition \ref{prop:borned} also holds for $\P\left( N_{n}d(\ph_n, p) \geq \delta | N_{n}>0\right)$: in fact, as 
\begin{multline*}
1=\E\left[W^{\lambda(z)}_n\right] = \E\left[W^{\lambda(z)}_n| N_{n}>0 \right] \P(N_{n}>0) + \E\left[W^{\lambda(z)}_n| N_{n}=0 \right] \P(N_{n}=0) \\= \E\left[W^{\lambda(z)}_n| N_{n}>0 \right] \P(N_{n}>0) + 1- \P(N_{n}>0),
\end{multline*}
it follows that $ \E\left[W^{\lambda(z)}_n| N_{n}>0 \right] =1$ and starting from Equation \ref{eq:umarkov}, the proof can be rewritten conditionally on $\{N_{n}>0\}$; this leads to:
$$\P\left( N_{n}d(\ph_n, p) \geq \delta | N_{n}>0 \right) \leq 2e\left\lceil \delta\log(n) \right\rceil \exp(-\delta).
 $$ 
However, in general no such result can be proved  for $\P\left( N_{n}d(\ph_n, p) \geq \delta | N_{n}>k \right)$ for positive values of $k$.
\end{rem}

To proceed, we need the following lemma:

\begin{lemma}\label{lem:devent}
For any probability distributions $P$ and $Q$ on the finite alphabet $A$,
 \begin{equation*}
D(P; Q) \leq \sum_{x\in A}d\left(P(x);Q(x)\right)\,.
\end{equation*}
\end{lemma}
\begin{proof}
\begin{eqnarray*}
\sum_{x\in A}d\left(P(x);Q(x)\right) -D(P;Q) &=& \sum_{x\in A}\left( 1-P(x) \right)\log \frac{1-P(x)}{1-Q(x)}\\
& = & \left(|X|-1\right) \sum_{x\in A}\frac{1-P(x)}{|A|-1}\log\left( \frac{(1-P(x)/(|A|-1)}{(1-Q(x)/(|A|-1)} \right)\\
& \geq  & 0
\end{eqnarray*}
because the sum in the next-to-last line is the Kullback-Leibler divergence between the probability distributions $R$ and $S$ defined on $A$ by:
\[R(x) = \frac{1-P(x)}{|A|-1} \quad\hbox{and}\quad S(x) = \frac{1-Q(x)}{|A|-1}\;.\]
\end{proof}
\begin{rem}
Obviously, this lemma is suboptimal for $|A|=2$ by a factor $2$.
For larger alphabets, it does not appear possible to improve on this bound for all $P$ and $Q$.
\end{rem}

We are now in position to state the deviation result we use in order to upper bound the probability of over-estimation: 
\begin{theorem}\label{th:borneD}
 Let $k$ be a positive integer and let $w\in A^k$. 
   Then, for any $\delta>0$
$$
\P\left[N_{n}(w)D\left(\ph_n\left(\cdot | w\right) ; p\left(\cdot | w\right)\right) > \delta\right] \;\leq\; 2e\left( \delta\log(n)+|A| \right)\exp \left( -\frac{\delta}{|A|} \right).$$
\end{theorem}
\begin{proof}
By combining  Lemma \ref{lem:devent} and Proposition~\ref{prop:borned}, we get

\begin{eqnarray*}
\P\left[N_{n}(w)D\left(\ph_n\left(\cdot | w\right) ; p\left(\cdot | w\right)\right) > \delta\right] 
& \leq & \P\left[\sum_{b\in A} N_{n}d\left(\ph_n\left(b | w\right) ; p\left(b | w\right)\right) > \delta\right]  \\
& \leq & \sum_{b\in A}\P\left[ N_{n}d\left(\ph_n\left(b | w\right) ; p\left(b | w\right)\right) > \frac{\delta}{|A|}\right]  \\
&\leq & 2|A|e\left\lceil \frac{\delta}{|A|}\log(n) \right\rceil \exp\left(-\frac{\delta}{|A|}\right)\\
&\leq & 2|A|e\left(\frac{\delta}{|A|}\log(n) +1\right) \exp\left(-\frac{\delta}{|A|}\right)\\\
&=&2e\left( \delta\log(n)+|A| \right)\exp \left( -\frac{\delta}{|A|} \right).
\end{eqnarray*}
\end{proof}

\begin{rem}\label{rem:borneDcond}
It follows from Remark \ref{rem:bornedcond} that the following variant of Theorem \ref{th:borneD} holds:
$$\P\left[N_{n}(w)D\left(\ph_n\left(\cdot | w\right) ; p\left(\cdot | w\right)\right) > \delta\, |\, N_{n}(w)>0\right] \;\leq\; 2e\left( \delta\log(n)+|A|-1 \right)\exp \left( -\frac{\delta}{|A|-1} \right).$$
\end{rem}

%%%%%%%%%%%%%%%%%%%%%%%%%%
\section*{Appendix B: Exponential inequalities for weak dependent processes}
\label{sec:exp}
\renewcommand{\thesection}{B}
%%%%%%%%%%%%%%%%%%%%%%%%%%

In this section we state some results providing  exponential inequalities  for processes satisfying Assumption~\ref{ass:phi} and prove two lemmas that are useful in the proof of Theorem~\ref{th:underestimation}. 
The first result is a version of Theorem~3.1 in \cite{galves2008} that we state under weaker conditions, given by Assumption~\ref{ass:phi}. 

\begin{prop}\label{theo:galves}
Assume the process $\{X_t\colon t\in\Z\}$ satisfies Assumption~\ref{ass:phi}. Then for any 
$w\in A^*$, any $a\in A$ and any $t>0$ the
following inequality holds
\begin{equation*}\label{Nn2}
\P(\,|N_n(w,a)-np(wa)|\,>\,t\,)\,\leq \,e^{\alpha_0/8e^2(|A|\beta+2\alpha_0)} 
\exp \Bigl(\,\frac{-t^2}{|wa|n}\,\Bigr)\,.
\end{equation*}
\end{prop}

\begin{proof}
Theorem~3.1 in \cite{galves2008} was proven for a process satisfying a stronger non-nullness hypothesis than our
Assumption~\ref{ass:phi}, namely that $\inf_{w\in T_0} \{p(a|w)\} > 0$  for any $a\in A$. But the proof of the
theorem is based on results obtained in \cite{comets2002} and \cite{DP} that also hold for
processes satisfying our weaker assumption. Moreover, the upper bound in Theorem~3.1 in \cite{galves2008} 
depends on the coefficient 
\[
\alpha:=\sum_{k\geq 0} (1-\alpha_k)\,,
\]
where for $k\geq 1$
\[
\alpha_k := \inf_{u\in A^k}\;\sum_{a\in A}\; \inf_{x_{-\infty}^{-1}} p(a|x_{-\infty}^{-1}u)\,.
\]
But it can be shown that for any $k\geq 1$ we have $1-\alpha_k\leq |A|\beta_k$, as noted by \cite{csis-tal2010} in their proof of Lemma~3.  Therefore $\alpha \leq |A| \beta + \alpha_0$
and Theorem~3.1 in \cite{galves2008} takes the form of Proposition~\ref{theo:galves}.
\end{proof}
As a consequence of this result we have the following lemma, proven in  \cite[Corollary~A.7]{leonardi2009}. 

\begin{lemma}\label{lemma:leo}
Assume the process $\{X_t\colon t\in\Z\}$ satisfies Assumption~\ref{ass:phi}. Then for any 
$w\in A^*$, any $a\in A$ and any $t>0$ the
following inequality holds
\[
\P\bigl(|\hat{p}_n(a|w)-p(a|w)|\,>\,t\bigl)
\,\leq\,e^{\alpha_0/32e^2|A|^2(|A|\beta+2\alpha_0)}(|A|+1)\exp \Bigl(  
\frac{-nt^2p(w)^2}{|w|+1}\Bigl)\,.
\]
\end{lemma}

Now, we prove Lemmas \ref{lemma:Nn} and \ref{lemma:div} below. These  two results are useful in the proof of Theorem~\ref{th:underestimation}.

\begin{lemma}\label{lemma:Nn}
Assume the process $\{X_t\colon t\in\Z\}$ satisfies Assumption~\ref{ass:phi}. Then for any 
$w\in A^*$ and any $t>0$ such that $t < np(w)$  we have
\begin{equation}\label{bound:Nn}
\P(\,N_n(w) \,\leq\, t\,)\,\leq \,e^{\alpha_0/8e^2|A|^2(|A|\beta+2\alpha_0)} |A|
\exp \Bigl(\frac{- n\, [p(w)-\frac{t}{n}]^2}{|w|+1}\Bigr)\,.
\end{equation}
\end{lemma}

\begin{proof}
 Using that $N_n(w)=\sum_{a\in A} N_n(w,a)$, $p(w)=\sum_{a\in A} p(wa)$ and $t-np(w)<0$ we have that 
  \begin{align*}
    \P(N_n(w) \,\leq\, t) \, &= \, \P\Bigl( \;\sum_{a\in A}[N_n(w,a) - 
    np(wa)]\, \leq\, t - np(w)\Bigr)\\
    \, &\leq \, \sum_{a\in A} \,\P( | N_n(w,a) - np(wa)| 
    \,\geq\, \frac{np(w) - t }{|A|})
  \end{align*}
  Using Theorem~\ref{theo:galves} we can bound above the right hand side of
  the last inequality by
  \[
  e^{\alpha_0/8e^2(|A|\beta+2\alpha_0)}\, |A|\, \exp \bigl[- \,\frac{[np(w)-t]^2}{|A|^2(|w|+1)n}\bigr].
  \]
This implies the bound in (\ref{bound:Nn}).
\end{proof}

\begin{lemma}\label{lemma:div}
Assume the process $\{X_t\colon t\in\Z\}$ satisfies Assumption~\ref{ass:phi}. Let $u,w\in A^*$ and $b\in A$ such that 
$p(b|u)- p(b|w)> 0$. Then,  for any $t< [p(b|u)- p(b|w)]^2/8 $ we have that 
\[
\P\bigl(D(\hat p_n(\cdot|u) ; \hat p_n(\cdot|w) ) \leq t\bigl)
\,\leq\,2 e^{\alpha_0/32e^2|A|^2(|A|\beta+2\alpha_0)} (|A|+1)\exp\Bigl[ -n\frac{t}{2}\min\Bigl( \frac{p(w)^2}{|w|+1}, \frac{p(u)^2}{|u|+1}  \Bigr)   \Bigr]\,.
\]
\end{lemma}
\begin{proof}
By Pinsker's inequality  (see, e.g, \cite[SectionA.2]{lugosi-book} for a proof) we have that 
\begin{align}\label{pins}
D(\hat p_n(\cdot|u) ; \hat p_n(\cdot|w) )\;&\geq \; \frac12 \Bigl[\,\sum_{a\in A}
|\hat p_n(a|u)  - \hat p_n(a|w)|\,\Bigr]^2\notag\\
&\ge\;\frac12 \bigl(\hat p_n(b|u)  - \hat p_n(b|w)\bigr)^2\,.
\end{align}
Now, set $\nu = \frac18[p(b|u)-p(b|w)]^2$ and define the events 
\begin{equation}\label{cn}
C_{n,\nu}^{b,w,u} = \{X_1^n\colon  |\hat p_n(b|u) - p(b|u)| \leq \sqrt{\nu/2}\}\,\cap \,\{X_1^n\colon  |\hat p_n(b|w) - p(b|w)| \leq \sqrt{\nu/2}\}\,.
\end{equation}
 Then, if $t< \nu$
 we have that the event
\begin{equation*}
\{X_1^n\colon D(\hat p_n(\cdot|u) ; \hat p_n(\cdot|w) ) \leq t \} \cap C_{n,\nu}^{b,w,u} = \emptyset\,.
\end{equation*}
To see this note that by (\ref{pins}), if (\ref{cn}) holds then  
\begin{align*}
D(\hat p_n(\cdot|u) ; \hat p_n(\cdot|w) )\;&\geq \;\frac12\, \Bigl[ \,(p(b|u) - \sqrt{\nu/2}) - (p(b|w)+ \sqrt{\nu/2}) \Bigr]^2 =
  \;\nu \; > \; t\,.
\end{align*}
Therefore, 
using the bounds in Lemma~\ref{lemma:leo} we obtain for any $t < \nu$ that
\begin{align*}
\P\bigl(D(\hat p_n(\cdot|u) ; \hat p_n(\cdot|w) ) \leq t\bigl)
\;&\leq\;  \P\bigl(  |\hat p_n(b|u) - p(b|u)| \geq \sqrt{\nu/2}\, \bigl) \; +\;  \P\bigl(  |\hat p_n(b|w) - p(b|w)| \geq\sqrt{\nu/2} \,\bigl)\\
&\leq\; 2 e^{\alpha_0/32e^2|A|^2(|A|\beta+2\alpha_0)}(|A|+1) \exp\Bigl[ -n\,\frac{\nu}{2}\min\Bigl( \frac{p(w)^2}{|w|+1},\frac{p(u)^2}{|u|+1}  \Bigr)   \Bigr]\,.
\end{align*}
\end{proof}

\section*{Acknowledgments}

This work was supported by Fapesp (grant 2009/09411-8) and USP-COFECUB (grant 2009.1.820.45.8).
F.L. is partially supported by a CNPq fellowship (grant 302162/2009-7).

\bibliographystyle{plain}
\bibliography{references}

\begin{thebibliography}{10}

\bibitem{barron1998}
A.~Barron, J.~Rissanen, and B.~Yu.
\newblock The minimum description length principle in coding and modeling.
\newblock {\em IEEE Trans. Inform. Theory}, 44(6):2743--2760, 1998.
\newblock Information theory: 1948--1998.

\bibitem{bejerano2001a}
G.~Bejerano and G.~Yona.
\newblock Variations on probabilistic suffix trees: statistical modeling and
  prediction of protein families.
\newblock {\em Bioinformatics}, 17(1):23--43, 2001.

\bibitem{breiman:al:84:cart}
L.~Breiman, J.H. Friedman, R.A. Olshen, and C.J. Stone.
\newblock {\em Classification and regression trees}.
\newblock Wadsworth Statistics/Probability Series. Wadsworth Advanced Books and
  Software, Belmont, CA, 1984.

\bibitem{buhlmann1999}
P.~B{\"u}hlmann and A.~J. Wyner.
\newblock {Variable length Markov chains}.
\newblock {\em Ann. Statist.}, 27:480--513, 1999.

\bibitem{busch2009}
J.~R. Busch, P.~A. Ferrari, A.~G. Flesia, R.~Fraiman, S.~P. Grynberg, and
  F.~Leonardi.
\newblock Testing statistical hypothesis on random trees and applications to
  the protein classification problem.
\newblock {\em Annals of applied statistics}, 3(2), 2009.

\bibitem{lugosi-book}
N.~Cesa-Bianchi and G.~Lugosi.
\newblock {\em Prediction, learning, and games}.
\newblock Cambridge University Press, Cambridge, 2006.

\bibitem{comets2002}
F.~Comets, R.~Fern\'andez, and P.~Ferrari.
\newblock Processes with long memory: Regenerative construction and perfect
  simulation.
\newblock {\em Ann. Appl. Probab.}, 12(3):921--943, 2002.

\bibitem{csiszar2002}
I.~Csisz{\'a}r.
\newblock Large-scale typicality of {M}arkov sample paths and consistency of
  {MDL} order estimators.
\newblock {\em IEEE Trans. Inform. Theory}, 48(6):1616--1628, 2002.
\newblock Special issue on Shannon theory: perspective, trends, and
  applications.

\bibitem{csiszar2006}
I.~Csisz{\'a}r and Z.~Talata.
\newblock Context tree estimation for not necessarily finite memory processes,
  via {BIC} and {MDL}.
\newblock {\em IEEE Trans. Inform. Theory}, 52(3):1007--1016, 2006.

\bibitem{csis-tal2010}
I.~Csisz{\'a}r and Z.~Talata.
\newblock On rate of convergence of statistical estimation of stationary
  ergodic processes.
\newblock {\em Information Theory, IEEE Transactions on}, 56(8):3637 --3641,
  2010.

\bibitem{DP}
J.~Dedecker and C.~Prieur.
\newblock New dependence coefficients. examples and applications to statistics.
\newblock {\em Probab. Theory Relatated Fields}, 132:203--236, 2005.

\bibitem{duarte2006}
D.~Duarte, A.~Galves, and N.L. Garcia.
\newblock Markov approximation and consistent estimation of unbounded
  probabilistic suffix trees.
\newblock {\em Bull. Braz. Math. Soc.}, 37(4):581--592, 2006.

\bibitem{fernandez2002}
R.~Fern{\'a}ndez and A.~Galves.
\newblock Markov approximations of chains of infinite order.
\newblock {\em Bull. Braz. Math. Soc.}, 33(3):295--306, 2002.

\bibitem{filippiCappeGarivier10KLUCRL}
S.~Filippi, O.~Capp{\'e}, and A.~Garivier.
\newblock Optimism in reinforcement learning based on kullback-leibler
  divergence.
\newblock {\em 48th Annual Allerton Conference on Communication, Control, and
  Computing}, 2010.

\bibitem{galves2009}
A.~Galves, C.~Galves, J.~Garcia, N.L. Garcia, and F.~Leonardi.
\newblock Context tree selection and linguistic rhythm retrieval from written
  texts.
\newblock {\em ArXiv: 0902.3619}, pages 1--25, 2010.

\bibitem{galves:garivier:gassiat:2010}
A.~Galves, A.~Garivier, and E.~Gassiat.
\newblock Data selection of context trees and classification.
\newblock Technical Report, 2010.

\bibitem{galves2008}
A.~Galves and F.~Leonardi.
\newblock {\em Exponential inequalities for empirical unbounded context trees},
  volume~60 of {\em Progress in Probability}, pages 257--270.
\newblock Birkhauser, 2008.

\bibitem{galves2006}
A.~Galves, V.~Maume-Deschamps, and B.~Schmitt.
\newblock Exponential inequalities for {VLMC} empirical trees.
\newblock {\em ESAIM Probab.~ Stat}, 12:43--45, 2008.

\bibitem{garivier2006}
A.~Garivier.
\newblock Consistency of the unlimited {BIC} context tree estimator.
\newblock {\em IEEE Trans. Inform. Theory}, 52(10):4630--4635, 2006.

\bibitem{COLT2011}
A.~Garivier and O.~Capp{\'e}.
\newblock The {KL-UCB} algorithm for bounded stochastic bandits and beyond.
\newblock In {\em 23rd Conf. Learning Theory (COLT)}, Budapest, Hungary, 2011.

\bibitem{garivier2008}
A.~Garivier and E.~Moulines.
\newblock On upper-confidence bound policies for non-stationary bandit
  problems, arxiv.org:0805.3415, 2008.

\bibitem{leonardi2009}
F.~Leonardi.
\newblock Some upper bounds for the rate of convergence of penalized likelihood
  context tree estimators.
\newblock {\em Brazilian Journal of Probability and Statistics},
  24(2):321--336, 2010.

\bibitem{massart2007}
P.~Massart.
\newblock {\em Concentration inequalities and model selection}, volume 1896 of
  {\em Lecture Notes in Mathematics}.
\newblock Springer, Berlin, 2007.
\newblock Lectures from the 33rd Summer School on Probability Theory held in
  Saint-Flour, July 6--23, 2003, With a foreword by Jean Picard.

\bibitem{neveu72}
J.~Neveu.
\newblock {\em Martingales \`a temps discret}.
\newblock Masson, 1972.

\bibitem{rissanen1983}
J.~Rissanen.
\newblock A universal data compression system.
\newblock {\em IEEE Trans. Inform. Theory}, 29(5):656--664, 1983.

\bibitem{talata2009}
Z.~Talata and T.~Duncan.
\newblock Unrestricted {BIC} context tree estimation for not necessarily finite
  memory processes.
\newblock In {\em Information Theory, 2009. ISIT 2009. IEEE International
  Symposium on}, pages 724--728, 28 2009-July 3 2009.

\bibitem{willems1995}
F.M.J. Willems, Y.M. Shtarkov, and T.J. Tjalkens.
\newblock The context-tree weighting method: Basic properties.
\newblock {\em IEEE Trans. Inf. Theory}, 41(3):653--664, 1995.

\end{thebibliography}
\end{document}